\documentclass[11pt]{article}
\usepackage[letterpaper,hmargin=1.5in,vmargin=1.75in]{geometry}

\usepackage{graphicx}
\usepackage{amsmath}
\usepackage{amsfonts}
\usepackage{amssymb}
\usepackage{latexsym}
\usepackage{amsthm}

\newtheorem{Prop}{Proposition}
\newtheorem{Clm}{Claim}

\begin{document}

\title{Strong pseudoprimes to the first 9 prime bases}
\author{ Yupeng Jiang and Yingpu Deng\\\\
Key Laboratory of Mathematics Mechanization, \\Academy of
Mathematics and Systems Science,\\ Chinese Academy of Sciences,
Beijing, China, 100190\\ \{jiangyupeng,dengyp\}@amss.ac.cn}
\date{}
\maketitle
\begin{abstract}
Define $\psi_m$ to be the smallest strong pseudoprime to the first $m$ prime bases. The exact value of $\psi_m$ is known for $1\le m \le 8$. Z. Zhang have
found a 19-decimal-digit number $Q_{11}=3825\,12305\,65464\,13051$ which is a strong pseudoprime to the first 11 prime bases and he conjectured that
$$\psi_9=\psi_{10}=\psi_{11}=Q_{11}.$$ We prove the conjecture by algorithms.
\end{abstract}

\section{Introduction}

If $n$ is prime, in view of Fermat's little theorem, the congruence $$a^{n-1}\equiv 1 \mod n$$ holds for every $a$ with gcd($a,n$)=1. There are composite numbers also satisfying the congruence. Such an odd composite number $n$ is called a pseudoprime to base $a$ (psp($a$) for short). Moreover for an odd prime $n$, let $n-1=2^{s}d$ with $d$ odd, we have $$a^d\equiv 1 \mod n$$ or $$a^{2^kd}\equiv -1 \mod n$$ for some $k$ satisfying $0\le k< d$. If a composite number $n$ satisfies these two equations, we call $n$ a strong pseudoprime to baes $a$ (spsp($a$) for short). This is the basic of Rabin-Miller test\cite{MR}.

Define $\psi_m$ to be the smallest strong pseudoprime to all the first $m$ prime bases. If $n<\psi_m$, then only $m$ strong pseudoprime tests are needed to find out whether $n$ is prime or not. If we know the exact value of $\psi_m$, then for integers $n<\psi_m$, there is a deterministic primality testing algorithm which is easier to understand and also faster than ever known other tests. The exact value of $\psi_m$ for $1\le m\le 8$ is known\cite{J,PSS}.
\begin{eqnarray*}
\psi_1 &=& 2047\\
\psi_2 &=& 1373653\\
\psi_3 &=& 25326001\\
\psi_4 &=& 32150\,31751\\
\psi_5 &=& 215\,23028\,98747\\
\psi_6 &=& 347\,47496\,60383\\
\psi_7 &=& 34155\,00717\,28321\\
\psi_8 &=& 34155\,00717\,28321
\end{eqnarray*}
In paper \cite{J}, Jaeschke also gave upper bounds for $\psi_9,\ \psi_{10},\ \psi_{11}$. These bounds were improved by Z. Zhang for several times
and finally he conjectured that
\begin{eqnarray*}
\psi_9=\psi_{10}=\psi_{11}=Q_{11}&=&3825\,12305\,65464\,13051\\
&=&149491\cdot747451\cdot34233211
\end{eqnarray*}
Zhang also gave upper bounds and conjectures for $\psi_m$, with $12\le m\le 20$ (see \cite{ZH1,ZH2,ZHT}).

In this paper, we develop several algorithms to get the following conclusion.
\begin{Clm}
$\psi_9=\psi_{10}=\psi_{11}=Q_{11}=3825\,12305\,65464\,13051.$
\end{Clm}

This article is organized like this. In \S 2 we give notations and basic facts needed for our algorithms. In \S 3 we get the properties of primes up to $\sqrt{Q_{11}}$ which give us much information to design our algorithm. Just as in \cite{J}, we consider the number of prime divisors of the testing number.
Let $n=p_1\cdot p_2\dots p_{t}$. In \S 4 we consider $t\ge5$ and $t=4$ respectively, \S 5 for $t=3$ and \S 6 for $t=2$. In \S 7 we get our conclusion and give the total time we need for our algorithms.

\section{Foundations of algorithms}

In this section, we give the foundations for our algorithm. Let $p$ be a prime, $a$ is an integer with gcd($a, p$)=1, denote the smallest positive integer $e$
such that $a^e\equiv 1\mod p$ by $Ord_{p}(a)$. For example, we have $Ord_{5}(2)=4$. Moreover for any integer $n$, if $n=p^en'$ with gcd($n, n'$)=1, we denote $e$ by $Val_p(n)$. In this article, we only use $Val_p(n)$ for $p=2$, we write $Val(n)$ by abbreviation. For $v\in \mathbb{Z}^n$, $v=(a_1, \dots, a_n)$ with all gcd($a_i, p$)=1
we define $$\sigma_{p}^{v}=(Val(Ord_p(a_1)), \dots, Val(Ord_p(a_n))).$$ If $n$ is a pseudoprime (or strong pseudoprime) for all the $a_i$s, we denote it by psp($v$) (or spsp($v$)).

We need to check all odd integers less than $Q_{11}$ to see if there are strong pseudoprimes to the first nine bases.
First we are going to exclude the integers having square divisors. If $n$ is a psp($a$) and $p^2|n$ for some prime $p$,
then we have $$a^{n-1}\equiv 1\mod p^2.$$ also $$a^{p(p-1)}\equiv 1\mod p^2.$$ As gcd($p,n-1$)=1, we have $$a^{p-1}\equiv 1\mod p^2$$
For $a=2$ and 3, $$2^{p-1}\equiv 1\mod p^2, \qquad 3^{p-1}\equiv 1\mod p^2$$ These two equation do not hold simultaneously for any prime $p$ less than $3\cdot10^9$\cite{PSS}, which is greater than $\sqrt{Q_{11}}\approx 1.9\cdot10^9$, so we only need to consider squarefree integers.

Now we give the following important proposition(also see \cite{J}).
\begin{Prop}
Let $n=p_1\dots p_t$ with different primes $p_1, \dots, p_t$, $v=(a_1, \dots, a_m)$ with different integers such that gcd($a_i,p_j$)=1 for all $i=1,\dots,m,\ j=1,\dots,t$. Then $n$ is an spsp(v) iff $n$ is a psp(v) and $\sigma_{p_1}^{v}=\dots=\sigma_{p_t}^{v}$.
\end{Prop}
\begin{proof}
Let $n-1=2^sd$ with $d$ odd. By Chinese Remainder Theorem $$a^{2^kd}\equiv -1 \mod n\Longleftrightarrow a^{2^kd}\equiv -1\mod p_i$$ for all $1\le i\le t$, so $Val(Ord_{p_i}(a))=k+1$ for all $i$. And $$a^{d}\equiv 1 \mod n\Longleftrightarrow a^{d}\equiv 1\mod p_i$$ for all $1\le i\le t$, so $Val(Ord_{p_i}(a))=0$ for all $i$. The proposition is an immediate consequence of the above argument.
\end{proof}
This is the main necessary condition that we use to find strong pseudoprimes. In our algorithm, $v=(2,3,5,7,11,13,17,19,23)$, For a given prime $p$, we need to find prime $q$ satisfying $\sigma_{p}^v=\sigma_{q}^v$. A problem we have to face is that there are too many candidates of $q$, so we need another proposition(also see \cite{J}).
\begin{Prop}
For primes $p,q$, if $Val(p-1)=Val(q-1)$ and $\sigma_p^{(a)}=\sigma_q^{(a)}$, then the Legendre symbol $(\frac{a}{p})=(\frac{a}{q})$.
\end{Prop}
\begin{proof}
This follows from $$\sigma_p^{(a)}=Val(p-1)\Longleftrightarrow (\frac{a}{p})=-1.$$
\end{proof}
Notice that if $p\equiv q\equiv 3 \mod 4$ in the above proposition, the inverse is also true. This is important and then we can use Chinese Remainder Theorem to reduce candidates. We'll give details in the following sections.

\section{Primes up to $\sqrt{Q_{11}}$}

From now on, we fix $v=(2,3,5,7,11,13,17,19,23)$. If $n$ is a psp($v$) and prime $p|n$, as $a^{n-1}\equiv1\mod p$, then $$Ord_p(a)|(n-1),\qquad a=2,3,5,7,11,13,17,19,23.$$ Define $\lambda_p$ to be the least common multiple of the nine orders, then we have $$\lambda_p|(n-1),\qquad \lambda_p|(p-1).$$ This point is helpful when designing our algorithms. Let $\mu_p=(p-1)/\lambda_p$, we develop an algorithm to calculate $\mu_p$ for $p$ up to $\sqrt{Q_{11}}$. It takes about 15 hours and find that $\mu_p$ is very small. We tabulate our results as following.

In the table, for each value of $\mu_p$, we give the first and last several primes. There are two rows with $\mu_p=2$, one for $p\equiv 3\mod 4$ and the other for $p\equiv 1\mod 4$. The binary row is for primes $p$ with $$p\equiv 1\mod 4,\qquad\sigma_p^{v}\in\{0,1\}^9.$$ Since $(\frac{2}{p})=-1$ for $p\equiv 5\mod 8$, in the second $\mu_p=2$ row all $p$ are in the residue class $1\mod 8$. For the same reason, in the binary row also with $p\equiv 1\mod 8$, as there is no prime with $\mu_p\ge 8$, all primes in binary row are $9\mod 16$ and with $\mu_p=4$. In the last column, we give the total number of each kind of primes.
\begin{center}
\begin{tabular}{|c|l|c|}
\multicolumn{3}{c}{$\mu_p$ for $p$ up to $\sqrt{Q_{11}}$ }\\[5pt]
\hline
$\mu_p$                              & \multicolumn{1}{c|}{primes}                      & total     \\ \hline
\raisebox{-7pt}[0pt][0pt]{2}         & 18191,\ 31391,\ 35279,\ 38639,\ 63839,\ 95471,   &           \\
                                     & 104711,\ 147671,\dots,\ 1955593559,\ 1955627519, & 93878     \\
\raisebox{7pt}[0pt]{$p\equiv 3(4)$}  & 1955645831,\ 1955687159,\ 1955728199             &           \\ \hline
\raisebox{-7pt}[0pt][0pt]{2}         & 87481,\ 185641,\ 336361,\ 394969,\ 483289,       &           \\
                                     & 504001,\ 515761,\dots,\ 1955712529,\ 1955713369, & 91541     \\
\raisebox{7pt}[0pt]{$p\equiv 1(4)$}  & 1955740609,\ 1955743729,\ 1955760361             &           \\ \hline
                                     & 4775569,\ 5057839,\ 5532619,\ 7340227,\ 7561207  &           \\
3                                    & 8685379,\ 9734161,\dots,\ 1953162751,            & 2226      \\
                                     & 1953185551,\ 1954279519,\ 1955425393             &           \\ \hline
                                     & 25433521,\ 120543721,\ 129560209,\ 138156769,    &           \\
4                                    & 148405321,\ 174865681,\dots,\ 1838581369,        & 111       \\
                                     & 1867026001,\ 1892769649,\ 1918361041             &           \\ \hline
                                     & 650086271,\ 792798571,\ 858613901,               &           \\
\raisebox{6.5pt}[0pt]{5}             & 1794251801,\ 1820572771,\ 1947963301             & \raisebox{6.5pt}[0pt]{6}   \\ \hline
6                                    & 1785200041                                       & 1         \\ \hline
7                                    & 945552637                                        & 1         \\ \hline
                                     & 120543721,\ 148405321,\ 200893081,\ 224683369,   &           \\
binary                               & 421725529,\ 481266361,\dots,\ 1717490329,        & 45        \\
                                     & 1810589881,\ 1828463641,\ 1838581369             &           \\ \hline
\end{tabular}
\end{center}

\section{$t\ge5$ and $t=4$}

As from the above, we only need to consider squarefree integers. we always denote $n=p_1\dots p_t$ with $p_1<\dots<p_t$. In this section, we are going to exclude the two cases when $t\ge5$ and $t=4$.

\subsection{$t\ge5$}
For $p$ up to $[\sqrt[5]{Q_{11}}]=5206$, let $S_p$ be the set of all primes $q$ with $\sigma_q^v=\sigma_p^v$, and denote $k$th element in $S_p$ by $s_{p,k}$ in ascending order. Our algorithm puts out the first $l$ elements of $S_p$ with $l>5$ and $$\prod_{i=1}^5s_{p,i} \le Q_{11},\qquad(\prod_{i=1}^4s_{p,i})s_{p,l}>Q_{11}.$$ It takes less than 22 seconds and puts out six sequences. We give our result in the following table.
\begin{center}
\begin{tabular}{|l|c|c|}
\multicolumn{3}{c}{sequence with equal $\sigma_p^v$}\\[5pt] \hline
                                                     & $\sigma_p^v$        & No.\\ \hline
167, 3167, 11087, 14423, 21383, 75407                & (0,0,1,0,0,1,1,0,1) & 1  \\ \hline
263, 1583, 8423, 9767, 12503, 18743, 50423,          &                     &    \\
54623, 106367, 127247                                & \raisebox{7pt}[0pt]{(0,0,1,1,0,0,0,1,0)} & \raisebox{7pt}[0pt]{13} \\ \hline
443, 4547, 5483, 8243, 19163, 26987, 42683           & (1,0,1,1,1,0,0,1,1) & 2  \\ \hline
463, 1087, 13687, 17383, 25447, 37447                & (0,1,1,1,1,1,0,1,1) & 1  \\ \hline
479, 4919, 5519, 6599, 7559, 29399, 51719            & (0,0,0,0,0,1,1,1,0) & 4  \\ \hline
2503, 2767, 5167, 5623, 11887, 31543                 & (0,1,1,1,0,1,0,0,0) & 1  \\ \hline
\end{tabular}
\end{center}
At first glance we know $t>5$ is impossible, Then we check these six sequences if they can make up an spsp($v$) with 5 prime divisors. The last column is the number of integers with $t=5$ and less than $Q_{11}$ in each sequence. Our checking algorithm terminates in less than 0.1 second and finds no strong pseudoprime.

There are details about our algorithm needing to explain. Notice that when $p_1\equiv3 \mod4$, and finding $q$ with $\sigma_{p_1}^v=\sigma_q^v$, as the least binary prime is 120543721. In fact we only need to check $q\equiv 3\mod 4$. by proposition 2, consider $$(\frac{2}{p_1})=(\frac{2}{q}),\qquad(\frac{3}{p_1})=(\frac{3}{q}).$$ we only need to check $q\equiv p_1\mod 24$. also in $p_1\equiv3\mod 4$ case, we calculate the Lengedre symbol $(\frac{\cdot}{p_1})$ instead of $Val(Ord_{p_1}(\cdot))$.

\subsection{t=4}
For $t=4$, we first define $(p_1, p_2, p_3)$ to be a feasible 3-tuple if it satisfies $$p_1<p_2<p_3,\quad \sigma_{p_1}^v=\sigma_{p_2}^v=\sigma_{p_3}^v,\quad p_1p_2p_3^2<Q_{11}.$$ Our algorithm goes like this: for each $p_1$ up to $[\sqrt[4]{Q_{11}}]=44224$, find feasible 3-tuples $(p_1, p_2, p_3)$. As $\lambda_{p_i}|n-1$, for $i=1,2,3$. let $\lambda$ be the least common multiple of these three numbers, and $b=p_1p_2p_3$, then we have $$n=bp_4\equiv 1\mod \lambda.$$ If gcd$(b,\lambda)\ne 1$, it is impossible to have such $n$. If gcd$(b,\lambda)=1$, we need to check all $p_4$ with
$$p_3<p_4\le Q_{11}/b,\qquad p_4\equiv b^{-1}\mod \lambda$$ Our algorithm takes about 15 minutes, finding 88729 feasible 3-tuples and no spsp($v$) with $t=4$. As for $t=5$, when $p_1\equiv 3\mod 4$, we use Legendre symbol and $q\equiv p_1\mod 24$ to shorten our running time.

\section{$t=3$}
As above, we define feasible 2-tuple $(p_1, p_2)$ with $$p_1<p_2, \quad\sigma_{p_1}^v=\sigma_{p_2}^v,\quad p_1p_2^2<Q_{11}$$
Our algorithm is just as $t=4$ case, for each $p_1$ up to $[\sqrt[3]{Q_{11}}]=1563922$, find feasible 2-tuples $(p_1, p_2)$. Let $b=p_1p_2$ and $\lambda=\text{lcm}(\lambda_{p_1},\lambda_{p_2})$, then $\lambda |n-1$. If gcd$(b,\lambda)=1$, we check all $p_3$ with
$$p_2<p_3\le Q_{11}/b,\qquad p_3\equiv b^{-1}\mod \lambda.$$ We divide our algorithm into three parts according $p_1\equiv 3\mod 4$, $p_1\equiv 5\mod 8$ and $p_1\equiv 1\mod 8$, also we use Chinese Remainder Theorem to reduce candidates.

\subsection{$p_1\equiv3\mod 4$}
For $p_1\equiv3 \mod 4$, we first assume $p_2\equiv 3\mod 4$. as from \S 3, we know if $p_2\equiv1 \mod 4$, $p_2$ must be a binary prime and so $\mu_{p_2}=4$. There are only 111 such primes up to $\sqrt{Q_{11}}$, we'll check these numbers later. By proposition 2, we use the first 5 primes and $$(\frac{a}{p_1})=(\frac{a}{p_2}),\qquad a=2,3,5,7,11$$ reducing to 30 residue classes module $9240=8\cdot3\cdot5\cdot7\cdot11$.

\noindent{\bf{Example 1}}. For $p_1=31$, the first module 4 equaling 3 prime. Feasible 2-tuple $(31,p_2)$ must with $$p_2<[\sqrt{Q_{11}/31}]=351270645$$ If we do not have \S 3, we need to check all the odd number greater than 31,  there are about $1.7\cdot 10^8$ candidates. If we do it as for $t=5$ and 4, there are $1.4\cdot10^7$ candidates. For our method, there are only $30\cdot\frac{351270645}{9240}\approx1.1\cdot10^6$ candidates.

There is another trick we used. if $b=p_1p_2$ is less than $2\cdot10^6$, the correspond $\lambda$ may be too small. We do not find $p_3$ as the above describes. In fact, as $$n=bp_3\equiv b\mod {p_3-1}$$ and $$a^{n-1}\equiv a^{b-1}\equiv 1\mod p_3, \qquad a=2,3$$ We calculate gcd$(2^{b-1}-1, 3^{b-1}-1)$ then factor it to get the prime divisor which is greater than $p_2$ and less than $Q_{11}/b$. Without this trick, our algorithm run more than 24 hours and still din't terminate. When using the trick, the algorithm takes less than 5 hours. It gives 10524046 feasible 2-tuples and the single spsp($v$) $$Q_{11}=3825\ 12305\ 65464\ 13051= 149491\cdot747451\cdot34233211.$$

The following table gives all the 37 feasible 2-tuples with multiple less than $2\cdot10^6$, which can explains why the first case takes so long time.

\noindent{\bf{Example 2}}. Notice that for some $b$ the $\lambda$ is small. For $b=43\cdot9283=339169$, we need to check all $p_3$ with
$$9283<p_3<Q_{11}/b\approx1.1\cdot10^{13},\qquad p_3\equiv7771\mod9282$$
and for $b=571\cdot2851=1627921$, all $p_3$ with
$$p_2<p_3<Q_{11}/b\approx2.3\cdot10^{12},\qquad p_3\equiv2281\mod2580$$
These are really time-consuming.

\subsection{$p_1\equiv5\mod 8$}
If $p_1\equiv 5\mod 8$, as $(\frac{2}{p_1})=-1$, $Val(Ord_{p_1}(2))=2$, so for each $p_2$ with $\sigma_{p_2}^v=\sigma_{p_1}^v$, we must have $p_2\equiv1\mod4$. If $p_2\equiv 5\mod 8$, by proposition 2, we use the first 5 primes then $$(\frac{a}{p_2})=(\frac{a}{p_1}),\qquad a=2,3,5,7,11.$$ There are 30 residue classes module 9240.
If $p_2\equiv 1\mod 8$, for $p_2\equiv 1\mod 16$, we must have $\mu_{p_2}=4$, we'll check these numbers later. For $p_2\equiv 9\mod 16$,
we must have$$(\frac{a}{p_2})=1,\qquad a=2,3,5,7,11.$$ There are 30 residue classes module 18480. The total time for checking all $p_1$ up to 1563922 is about 10 hours and we find 522239 feasible 2-tuples with no spsp($v$).

\subsection{$p_1\equiv1\mod8$}
For $p_1\equiv1\mod8$, denote $e=Val(p_1-1)$ and $f=Val(\lambda_{p_1})$, then $f\le e$ and $$p_1\equiv 1+2^{e}\mod{2^{e+1}},\qquad p_2\equiv 1\mod {2^f}$$ for $\sigma_{p_2}^v=\sigma_{p_1}^v$. If $f=e$, then we consider two cases. For $p_2\equiv 1+2^{e}\mod{2^{e+1}}$, we have
$$(\frac{a}{p_2})=(\frac{a}{p_1}),\qquad a=2,3,5,7,11$$
\begin{center}
\begin{tabular}{|l|l|l|l|c|}
\multicolumn{5}{c}{feasible $(p_1,p_2)$ with $b<2\cdot10^6$}\\[5pt]\hline
b        & $p_1$ & $p_2$ & $\lambda$  & $\sigma_{p_1}^v$             \\ \hline
 685441  & 31   & 22111  & 22110    & ( 0, 1, 0, 0, 1, 1, 1, 0, 1 )  \\ \hline
 919801  & 31   & 29671  & 29670    & ( 0, 1, 0, 0, 1, 1, 1, 0, 1 )  \\ \hline
1267249  & 31   & 40879  & 204390   & ( 0, 1, 0, 0, 1, 1, 1, 0, 1 )  \\ \hline
 399169  & 43   & 9283   & 9282     & ( 1, 1, 1, 1, 0, 0, 0, 1, 0 )  \\ \hline
 703609  & 43   & 16363  & 114534   & ( 1, 1, 1, 1, 0, 0, 0, 1, 0 )  \\ \hline
1379569  & 43   & 32083  & 224574   & ( 1, 1, 1, 1, 0, 0, 0, 1, 0 )  \\ \hline
1487929  & 43   & 34603  & 242214   & ( 1, 1, 1, 1, 0, 0, 0, 1, 0 )  \\ \hline
1772761  & 43   & 41227  & 288582   & ( 1, 1, 1, 1, 0, 0, 0, 1, 0 )  \\ \hline
 741049  & 47   & 15767  & 362618   & ( 0, 0, 1, 0, 1, 1, 0, 1, 1 )  \\ \hline
1879201  & 47   & 39983  & 919586   & ( 0, 0, 1, 0, 1, 1, 0, 1, 1 )  \\ \hline
 117049  & 67   & 1747   & 19206    & ( 1, 1, 1, 1, 1, 1, 0, 0, 0 )  \\ \hline
1578721  & 67   & 23563  & 23562    & ( 1, 1, 1, 1, 1, 1, 0, 0, 0 )  \\ \hline
1354609  & 71   & 19079  & 667730   & ( 0, 0, 0, 1, 1, 1, 1, 0, 1 )  \\ \hline
 722929  & 79   & 9151   & 118950   & ( 0, 1, 0, 1, 0, 0, 1, 0, 0 )  \\ \hline
1272769  & 79   & 16111  & 209430   & ( 0, 1, 0, 1, 0, 0, 1, 0, 0 )  \\ \hline
 457081  & 83   & 5507   & 225746   & ( 1, 0, 1, 0, 0, 1, 0, 1, 0 )  \\ \hline
1391329  & 83   & 16763  & 687242   & ( 1, 0, 1, 0, 0, 1, 0, 1, 0 )  \\ \hline
1739929  & 83   & 20963  & 859442   & ( 1, 0, 1, 0, 0, 1, 0, 1, 0 )  \\ \hline
1652401  & 107  & 15443  & 818426   & ( 1, 0, 1, 1, 0, 0, 1, 0, 0 )  \\ \hline
1730689  & 139  & 12451  & 286350   & ( 1, 1, 0, 0, 0, 0, 1, 1, 1 )  \\ \hline
1790881  & 163  & 10987  & 296622   & ( 1, 1, 1, 1, 1, 1, 1, 1, 1 )  \\ \hline
 528889  & 167  & 3167   & 262778   & ( 0, 0, 1, 0, 0, 1, 1, 0, 1 )  \\ \hline
1851529  & 167  & 11087  & 920138   & ( 0, 0, 1, 0, 0, 1, 1, 0, 1 )  \\ \hline
1892881  & 211  & 8971   & 62790    & ( 1, 1, 0, 1, 0, 0, 1, 0, 1 )  \\ \hline
1552849  & 229  & 6781   & 128820   & ( 2, 0, 1, 2, 1, 2, 0, 0, 2 )  \\ \hline
 416329  & 263  & 1583   & 207242   & ( 0, 0, 1, 1, 0, 0, 0, 1, 0 )  \\ \hline
 223609  & 311  & 719    & 111290   & ( 0, 0, 0, 0, 1, 0, 1, 1, 1 )  \\ \hline
1912849  & 331  & 5779   & 317790   & ( 1, 1, 0, 1, 1, 1, 0, 0, 1 )  \\ \hline
 825841  & 379  & 2179   & 45738    & ( 1, 1, 0, 1, 1, 1, 1, 0, 0 )  \\ \hline
 540409  & 439  & 1231   & 89790    & ( 0, 1, 0, 0, 0, 0, 1, 0, 1 )  \\ \hline
 503281  & 463  & 1087   & 83622    & ( 0, 1, 1, 1, 1, 1, 0, 1, 1 )  \\ \hline
 929041  & 503  & 1847   & 463346   & ( 0, 0, 1, 0, 0, 0, 1, 1, 0 )  \\ \hline
1627921  & 571  & 2851   & 2850     & ( 1, 1, 0, 1, 0, 0, 1, 1, 0 )  \\ \hline
1280449  & 787  & 1627   & 213006   & ( 1, 1, 1, 0, 0, 1, 1, 0, 0 )  \\ \hline
1616521  & 919  & 1759   & 268974   & ( 0, 1, 0, 1, 0, 0, 0, 1, 0 )  \\ \hline
1538161  & 1063 & 1447   & 255942   & ( 0, 1, 1, 0, 0, 0, 0, 0, 0 )  \\ \hline
1772521  & 1103 & 1607   & 884906   & ( 0, 0, 1, 1, 1, 1, 0, 1, 0 )  \\ \hline
\end{tabular}
\end{center}
There are 30 residue classes module $2^{e+1}\cdot1155$. For $p_2\equiv1+2^{e+1}\mod{2^{e+2}}$, we have
$$(\frac{a}{p_2})=1,\qquad a=2,3,5,7,11$$ 30 residue classer module $2^{e+2}\cdot1155$. The $p_2\equiv1\mod{2^{e+2}}$ case is left for the prime with $\mu_{p_2}=4$. If $f<e$, we only check $p_2\equiv p_1\mod 2^f$. In fact, according to \S 3, this only happens when $f=e-1$ and $\mu_{p_1}=2$. There are only 50 such primes up to 1563922. Our algorithm takes less than 100 minutes and finds 30728 feasible 2-tuples and no spsp($v$).

\subsection{$\mu_{p_2}=4$}
In the above three cases, we don't consider the case $\mu_{p_2}=4$. Now we assume $\mu_{p_2}=4$, as we also have $$p_1\ge 29, \qquad p_1p_2^2\le Q_{11}$$ So $p_2<363181490$. According \S 3, there only 12 primes under this bound. We check all of them and find no feasible 2-tuples. Until now we finish the $t=3$ case and find only one spsp($v$) $Q_{11}$. The total time is less than 17 hours.

\section{t=2}
For $t=2$, there is no need to define feasible 1-tuples. As $\lambda_{p_1}|n-1$ we have $$p_1<p_2\le Q_{11}/p_1, \qquad p_2\equiv 1\mod\lambda_{p_1}.$$ Since $\lambda_{p_1}$ is close to $p_1-1$, there are about $Q_{11}/(p_1)^2$ candidates for each $p_1$. When $p_1$ is small, there are too many. According the value of $p_1$, we divide into three parts.

\subsection{small and large $p_1$}
If $p_1<10^6$, we'll use the same method as for $t=3$, $p_1p_2<2\cdot10^6$. We have $$a^{p_1-1}\equiv a^{n-1}\equiv1\mod p_2,\qquad a=2,3$$ so we calculate gcd$(2^{p_1-1}-1,3^{p_1-1}-1)$ and factor it to get prime divisors $p_2$ with $$p_1<p_2\le Q_{11}/p_1.$$ Our algorithm takes about 9 hours and finds no spsp($v$).

For $p_1>10^8$, There are less than 380 candidates, we just run our algorithm as described at the beginning of this section. It takes about 18 hours and find no spsp($v$).

\subsection{$10^6<p_1<10^8$}
When $p_1$ is in this interval, we divide into three parts according to $p_1\equiv 3\mod4$, $p_1\equiv5\mod 8$ and $p_1\equiv 1\mod 8$. In each case, just as $t=3$ we use Chinese Remainder Theorem to reduce candidates. This time we use the first 6 primes.

For $p_1\equiv 3\mod 4$, $p_2$ with $\sigma_{p_2}^v=\sigma_{p_1}^v$. If $p_2\equiv 3\mod4$ then we have $p_2\equiv1 \mod \lambda_{p_1}$ and $$(\frac{a}{p_1})=(\frac{a}{p_2}),\qquad a=2,3,5,7,11,13.$$ If $p_2\equiv1\mod 4$, then we have $p_2\equiv 1\mod \lambda_{p_1}$ and
$$(\frac{a}{p_2})=1,\qquad a=2,3,5,7,11,13.$$ Our algorithm takes about 15 hours and finds no spsp($v$).

For $p_1\equiv 5\mod 8$, then $p_2\equiv 1\mod 4$. If $p_2\equiv 5\mod 8$ then we have $p_2\equiv 1\mod \lambda_{p_1}$ and
$$(\frac{a}{p_1})=(\frac{a}{p_2}),\qquad a=2,3,5,7,11,13.$$ If $p_2\equiv 1\mod 8$, then we have $p_2\equiv 1\mod \lambda_{p_1}$ and
$$(\frac{a}{p_2})=1,\qquad a=2,3,5,7,11,13.$$ Our algorithm takes about 15 hours and finds no spsp($v$).

For $p_1\equiv1 \mod 8$, denote $e=Val(p_1-1)$, $f=Val(\sigma_{p_1})$, then $f\le e$. If $f=e$, there are two cases. For $p_2\equiv 1+2^e\mod 2^{e+1}$, then
$p_2\equiv 1\mod \lambda_{p_1}$ and $$(\frac{a}{p_1})=(\frac{a}{p_2}),\qquad a=2,3,5,7,11,13.$$ For $p_2\equiv1\mod 2^{e+1}$, then $p_2\equiv 1\mod\lambda_{p_1}$ and
$$(\frac{a}{p_2})=1,\qquad a=2,3,5,7,11,13.$$
If $f<e$, we only use $p_2\equiv 1\mod \lambda_{p_1}$, Our algorithm takes about 16 hours and finds no spsp($v$).

We also run an algorithm for these cases without use Chinese Remainder Theorem, it took more than 10 days and didn't halt. So the Chinese Remainder Theorem is really helpful here. We need to be careful when writing our algorithm because gcd$(a,\lambda_{p_1})\ne 1$ for some $p_1$ and $a=2,3,5,7,11,13.$

Then we finish the $t=2$ case and find no strong pseudoprime to the first 9 primes.

\section{Conclusion}
Until now, we have checked all the odd composite numbers up to $Q_{11}$, and find only one strong pseudoprime $Q_{11}$ to the first 9 primes. As it is easy to check that $Q_{11}$ is also strong pseudoprime to the bases 29 and 31, we have our claim in \S 1. $$\psi_{9}=\psi_{10}=\psi_{11}=Q_{11}$$ So for an integer less than $Q_{11}$, only 9 strong pseudoprime tests are needed to judge its primality and compositeness. We use the software Magma and all algorithms are run in my PC(an Intel(R) Core(TM)2 Duo CPU E7500 @ 2.93GHz with 2Gb of RAM). The total time is about 105 hours.

\end{document}